\documentclass[a4paper]{amsart}
\usepackage{amsfonts}
\usepackage{amsmath}
\usepackage{mathrsfs}
\usepackage{graphicx}
\usepackage{amsmath,amsthm,amssymb,amscd}
\usepackage{color}
\usepackage{enumerate}
\usepackage[T1]{fontenc}
\usepackage[all]{xy}

\newtheorem{thm}{Theorem}[section]

\newtheorem{defn}[thm] {Definition}
\newtheorem{ex}  [thm]{Example}
\newtheorem{cor}[thm]{Corollary}
\theoremstyle{remark}
\newtheorem{rem} [thm]{Remark}

\theoremstyle{definition}

\DeclareMathOperator{\supp}{supp}

\DeclareMathOperator{\mesh}{mesh}
\DeclareMathOperator{\esssup}{esssup}
\DeclareMathAlphabet{\mathcalligra}{T1}{calligra}{m}{n}

\newcommand{\eps}{\varepsilon}
\newcommand{\R}{\mathbb{R}}

\newcommand{\htop}{h_{\text{top}}}

\newcommand{\tx}{\tilde{x}}
\newcommand{\ty}{\tilde{y}}

\begin{document}

\title[On entropy of $\Phi$-irregular and $\Phi$-level sets]{On entropy of $\Phi$-irregular and $\Phi$-level sets in maps with the shadowing property}
\author[M. Fory\'{s}-Krawiec]{Magdalena Fory\'{s}-Krawiec}
\address[M. Fory\'s-Krawiec]{
	National Supercomputing Centre IT4Innovations, Division of the University of Ostrava,
	Institute for Research and Applications of Fuzzy Modeling,
	30. dubna 22, 70103 Ostrava,
	Czech Republic}
\email{magdalena.forys@osu.cz}

\author[J. Kupka]{Jiri Kupka}
\address[J. Kupka]{National Supercomputing Centre IT4Innovations, Division of the University of Ostrava,
	Institute for Research and Applications of Fuzzy Modeling,
	30. dubna 22, 70103 Ostrava,
	Czech Republic}
\email{jiri.kupka@osu.cz}

\author[P. Oprocha]{Piotr Oprocha}
\address[P. Oprocha]{AGH University of Science and Technology, Faculty of Applied
	Mathematics, al.
	Mickiewicza 30, 30-059 Krak\'ow, Poland
	-- and --
	National Supercomputing Centre IT4Innovations, Division of the University of Ostrava,
	Institute for Research and Applications of Fuzzy Modeling,
	30. dubna 22, 70103 Ostrava,
	Czech Republic}
\email{oprocha@agh.edu.pl}

\author[X. Tian]{Xueting Tian}
\address[X. Tian]{School of Mathematical Science, Fundan University, Shanghai 200433, People's Republic of China}
\email{xuetingtian@163.com}
\maketitle

\section{introduction}
Studies on shadowing property have long tradition within theory of dynamical systems. They originated from studies of Bowen and Anosov from 1970s.
These early works brought evidence, that there are strong connections between shadowing property, entropy and ergodic measures
(e.g. see \cite{Bow1}, \cite{Bow2}), bringing motivation for intensive studies lasting last 50 years. It was also Bowen who introduced specification property, a strong mixing condition
that can be observed in mixing maps with shadowing property. This started a very interesting chapter in studies of tracing (pseudo) orbits. In particular, in \cite{Sig1} Sigmund showed that for dynamical systems with periodic specification property the set of ergodic measures is the complement of a set of first category in the space of invariant measures.
In fact, technique developed by Sigmund, which allows to approximate given measure by ergodic measure, is among standard tools nowadays.

It was also observed that regularity of dynamics guaranteed by ergodic theorem, is always complemented by quite ``wild'' or ``irregular'' dynamical behavior. Let us state more precisely what we mean by these irregularities.
For a continuous function $\Phi \in \mathcal{C}(X,\mathbb{R})$ and for every $x \in X$ we define the sum $\frac{1}{n}\sum_{i=0}^{n-1}\Phi(T^ix)$ as a \textit{Birkhoff average}. By the classical Birkhoff ergodic theorem this sum converges for a set of points of full $T$-invariant measure. Those points, for which the Birkhoff average converges are called \textit{$\Phi$-regular points}. The complementary set is called a~\textit{$\Phi$-irregular set} and is denoted by $I_{\Phi}(T)$. The set of all irregular points:
$$
I(T) = \bigcup_{\Phi\in\mathcal{C}(X,\mathbb{R})}I_{\Phi}(T)
$$
is called an \textit{irregular set}. Irregular points are also referred to as \textit{points with historical behavior}, which suggests that points for which Birkhoff average converges present just the average behavior of the system, while the irregular points are capturing the ''complete history'' of the system.
It is clear, that from the ergodic point of view irregular sets are negligible. However if we study their qualitative aspects of dynamics, it turns out their dynamical structure may be quite complicated and interesting. In some cases they may be dynamically even as  complex as the whole space. By \cite{FFW} we know that for symbolic dynamics $I_{\Phi}(T)$ is either empty or has full topological entropy, while in \cite{BSch} the authors showed that for a class of systems containing horseshoes and conformal repellers irregular sets carry full entropy and Hausdorff dimension. Studies of Olsen (e.g. see \cite{O1}, \cite{O2}, \cite{Ow} on deformations and empirical measures, together with the methodology from \cite{TV} motivated the authors of \cite{ChTS} to consider irregular points in systems with specification property. They showed that $I(T)$ has full topological entropy in such case, while the authors of \cite{LW} gave a topological characterization of $\Phi$-irregular set $I_{\Phi}(T)$ in systems with the specification property, by proving that it is either empty or residual.
Later, Thompson made an attempt to weaken the assumptions and proved in \cite{Thom} that for dynamical systems with almost specification the set $I_{\Phi}(T)$ is either empty or carries full topological entropy. In \cite{DOT} the authors consider (possibly not transitive) systems with the shadowing property and proved that $I(T)$ carries full topological entropy when nonempty.

In fact, shadowing property allows to look deeper in the topological structure of the system. In \cite{Mooth} and \cite{MoothO} the authors proved that both collections of uniformly recurrent points and regularly recurrent points are dense in the non-wandering set (that is, Toeplitz minimal systems are dense). Later, Li and Oprocha in \cite{LO} extended these results, showing that points whose orbit closure are odometers are dense in non-wandering set and, moreover, in transitive system with shadowing the collection of ergodic measures supported on odometers is dense in the space of invariant measures. They also proved that ergodic measures supported on Toeplitz systems may approximate other measures in weak*-topology and in terms of value of entropy. While in the case of specification we control only convergence of measures, shadowing property provides deeper insight into the structure of $\omega$-limit sets.

Successful approach in \cite{LO} provides motivation for the present work. We aim for better understanding of local dynamical structure of systems with the shadowing property.
The paper is organized as follows.

In the preliminary section we introduce some basic facts about dynamical systems with the shadowing property, measures and ergodic theory. In the next section we focus on $\Phi$-irregular sets in dynamical systems with the shadowing property.
We estimate value of entropy of $\Phi$-irregular set over chain recurrent set it intersect (see Theorem \ref{thmhtop})
and then use this result to express entropy of nonempty $\Phi$-irregular sets in terms of entropies of chain recurrent sets ( Corollary \ref{cor:CRclass} and \ref{CorSup}).
We also show that $\Phi$-irregular sets of full entropy are typical (Theorem~\ref{thmtyp}).
%
%
%
At the end, we study properties of $\Phi$-level sets (sets of points whose Birkhoff averages are in a given set) and relate it to entropies of some ergodic measures (see Theorem~\ref{LevelSetThm}). We also consider level sets with respect to reference measures, proving local analogs of result of L. Young in \cite{Young}, who considers the problem of large deviations in systems with the specification property (see Theorem \ref{thm:v-}).
While shadowing property provides slightly better control of tracing than specification, we cannot apply global ``gluing'' condition, provided by specification. It is why we have to work in neighborhoods of chain recurrent classes. Surprisingly, despite of local character of our considerations, some global estimates are obtained.

\section{preliminaries}

\subsection{Basic notions and definitions}

A \textit{dynamical system} is a pair $(X,T)$ consisting of a compact metric space $(X,d)$ and a continuous mapping $T:X\rightarrow X$. Let $x\in X$.  By a  \textit{trajectory} of the point $x$ we mean a sequence $\{T^nx\}_{n\geq 0}$, $n\in \mathbb N$, and by an \emph{orbit} of  $x$ we call the set of all iterations, that is the following set:
$$
\{T^nx: n \geq 0\}.
$$
For two integers $L>K\geq0$, by writing $x_{[K,L]}$ we mean a~\textit{block}, i.e.  a finite part of the trajectory of $x$:
$$
T^Kx, T^{K+1}x, \dots, T^{L-1}x, T^{L}x.
$$

For a finite set of indices $\Lambda\subseteq\{0,1,\dots,n-1\}$ and $x,y \in X$ we define a \textit{Bowen distance} of $x,y$ along $\Lambda$ by the following formula:
$$
d_{\Lambda}(x,y) = \max_{j\in \Lambda}\{d(T^{j}x,T^jy)\},
$$
and a \textit{Bowen ball} of radius $\eps>0$ centered in $x\in X$ as the following set:
$$
B_{\Lambda}(x,\eps) = \{y \in X: d_{\Lambda}(x,y)<\eps\}.
$$
In particular, when $\Lambda = \{0,1,\dots,n-1\}$ we denote $d_{\Lambda}(x,y)$ by $d_n(x,y)$ and $B_{\Lambda}(x,\eps)$ by $B_n(x,\eps)$. A set $B_n(x,\eps)$ is then called an $(n,\eps)$-Bowen ball.

\begin{defn}
A dynamical system $(X,T)$ is \emph{topologically transitive} if for every pair of nonempty open sets $U,V \subseteq X$ there exists an integer $M$ such that $T^M(U)\cap~V~\neq \emptyset$.
\end{defn}

\begin{defn}
Given $\delta$>0, a sequence $\{x_n\}_{n \in \mathbb{N}}\subseteq X$ is a \emph{$\delta$-pseudo-orbit} of $T$ if:
	$$
	d(Tx_n,x_{n+1})<\delta \text{ for all } n \in \mathbb{N}.
	$$
A finite $\delta$-pseudo-orbit $\{x_n\}_{n=0}^k$ is often called a \emph{$\delta$-chain} from $x_0$ to $x_k$.
\end{defn}

\begin{defn}
	A map $T:X\rightarrow X$ has a \emph{shadowing property} if for all $\eps>0$ there exists $\delta>0$ such that for any $\delta$-pseudo-orbit $\{x_n\}_{n\in \mathbb{N}}$ there exists a point $y \in X$ such that:
	$$
	d(T^ny,x_n)<\eps \text{ for all }n \in \mathbb{N}.
	$$
	We say that $\delta$-pseudo-orbit $\{x_n\}_{n \in \mathbb{N}}$ is \emph{$\eps$-shadowed} (resp. \emph{$\eps$-traced}) by an orbit of $y$.
\end{defn}

For a point $x \in X$ define its \textit{$\omega$-limit set} $\omega_T (x)$ as the set of limit points of the trajectory of $x$:
$$\omega_T(x) = \bigcap_{n=0}^{\infty} \overline{\{T^k(x): k\geq n\}}.$$
Let $Y^{\omega}$ denote the set of all points whose $\omega$-limit sets are  subsets of $Y$, that is:
$$
Y^{\omega}=\{x \in X: \omega_T(x)\subseteq Y\}.
$$
It is well known that $Y^{\omega}$ need not be compact.
\begin{defn}
	Let $x,y \in X$ and $\eps>0$. If there exist a sequence of points $\{x_i\}_{i=0}^n~\subseteq~X$ and an increasing sequence of positive integers $\{t_i\}_{i=0}^{n-1}$ such that:
	\begin{align*}
	x_0&=x,\\
	x_n&=y,\\
	d(T^{t_i}x_i,x_{i+1})&<\eps \text{ for } i=0,\dots,n-1,
	\end{align*}
	we say that $x$ is in a \emph{chain stable set} of $y$.
\end{defn}

If $x$ is in a chain stable set of $y$ and $y$ is in a chain stable set of $x$ we say that the points $x$ and $y$ are \textit{chain related}. Note that this relation is an equivalence relation. If $x$ is chain related with itself we say that $x$ is a \textit{chain recurrent point}. By $CR(X)\subseteq X$ we denote the set of chain recurrent points in $X$ and by a \textit{chain recurrent class} we call every equivalence class in $CR(X)$ given by the chain relation.

\begin{defn}
	A subset $Y \subseteq X$ is an internally chain transitive set if for all $x,y \in Y$ and for every $\eps>0$ there exist an $\eps$-chain from $x$ to $y$ consisting only of points from $Y$.
\end{defn}

\begin{rem}\label{omegaSub}
Since each chain recurrent class $Y\subseteq X$ is invariant, we always have $Y\subseteq Y^{\omega}$.
\end{rem}
The statement follows from the fact that every chain recurrent class is closed and $T$-invariant ($T(Y) \subseteq Y$).

\begin{defn}
	A point $x \in X$ is $\mu$-generic for a measure $\mu$ if for every continuous mapping $f:X\rightarrow \mathbb{R}$ the following condition holds:
	$$\lim_{n\rightarrow \infty}\frac{1}{n}\sum_{i=0}^{n-1}f(T^ix)=\int_Xfd\mu,$$
	or, equivalently,
	$$
	\frac{1}{n}\sum_{i=0}^{n-1}\delta_{T^ix}=\mu,
	$$
	where $\delta_y$ denotes the Dirac measure on $y$.
\end{defn}

From the above definition it follows that the orbit of every generic point is dense on the support of the measure provided that $x$ belongs to that support. In that case the orbit of the generic point has a~nonempty intersection with every open neighbourhood of that point, hence every generic point is also recurrent.

\subsection{Measures}

For a compact metric space $X$  by $\mathcal{B}$ we denote  the $\sigma$-algebra of Borel subsets of $X$. Let $\mathcal{M}(X)$ be the set of all Borel probability measures on the space $(X,\mathcal{B})$. The \textit{support} of a measure $\mu \in \mathcal{M}(X)$, denoted by $\supp(\mu)$ is the smallest closed subset $C\subset X$ such that $\mu(C)=1$.

For a dynamical system $(X,T)$ we say that a measure $\mu \in \mathcal{M}(X)$ is \textit{$T$-invariant} if $\mu(T^{-1}A)=\mu(A)$ for all $A \in \mathcal{B}$, and $\mu$ is \textit{ergodic} if the only Borel sets $B$ satisfying $T^{-1}B = B$ are sets of zero or full measure, i.e.  $\mu(B)=0$ or $\mu(B)=1$. By $\mathcal{M}_T(X)$ we denote  the set of all $T$-invariant measures and  the set of all ergodic measures on $X$ is denoted by $\mathcal{M}_e(X)$.

By the Riesz representation theorem we may look at $\mathcal{M}(X)$ as a compact metric space with the metric given by the weak$^*$ topology of the dual space $\mathcal{C}(X,\mathbb{R})$. To define the convergence in $\mathcal{M}(X)$, we say that a sequence of measures $\{\mu_n\}_{n \in \mathbb{N}}$ converges to a~measure $\mu \in \mathcal{M}(X)$ in the weak$^*$ topology if the following expression holds for every $\phi \in \mathcal{C}(X,\mathbb{R})$:
$$
\lim_{n\rightarrow\infty}\int\phi d\mu_n = \int \phi d\mu.
$$
Let $BL(X)$ be the set of all bounded Lipschitz real-valued functions on $X$. Note that $BL(X)$ is dense in $\mathcal{C}(X,\mathbb{R})$. Let $\|\varphi\|_{BL} = \|\varphi\|_{\infty}+\|\varphi\|_{L}$, where $\|\varphi\|_{\infty}$ is the supremum norm, and:
$$
\|\varphi\|_{L} = \sup\frac{|\varphi(x)-\varphi(y)|}{d(x,y)}<\infty.
$$
For a countable sequence $\{\varphi_n\}_{n \in \mathbb{N}} \subset \{\varphi \in BL(X): \|\varphi\|_{BL}\leq 1 \}$ and for measures $\mu,\nu~\in~\mathcal{M}(X)$ we define the following metric:
$$
d_{\mathcal{M}}(\mu,\nu) = \sum_{n=1}^{\infty}\frac{1}{2^n}\left|\int \varphi_n d\mu - \int \varphi_n d\nu\right|.
$$
Then $d_{BL}$ is a metric on $\mathcal{M}(X)$ and the topology induced by $d_{BL}$ coincides with the weak$^*$ topology.

\subsection{Metric entropy}

The idea of metric entropy has its motivation in Shannon's information theory. We will present the definition of so-called Kolmogorov-Sinai metric entropy. Let $(X,\mathcal{B},\mu)$ be a probability space, where $\mathcal{B}$ is a $\sigma$-algebra of subsets of $X$. Let $T:X\rightarrow X$ be a measure preserving transformation, i.e.  $T^{-1}A \in~\mathcal{B}$ and $\mu(A) = \mu(T^{-1}A)$ for all $A \in \mathcal{B}$. Let $\alpha = \{A_1,\dots,A_k\}$ be a finite partition of $X$. Put $T^{-i}\alpha = \{ T^{-i}A_1,\dots, T^{-i}A_k\}$. For two partitions $\alpha$ and $\beta$ denote $\alpha\vee\beta = \{A\cap B: A \in \alpha, B \in \beta\}$ and define the set $\bigvee_{i=0}^{n-1}T^{-i}\alpha$ as the collection of sets of the form $\{x: x \in A_{i_0},Tx \in A_{i_1},\dots, T^{n-1}x \in A_{i_{n-1}}\}$ for some $(i_0,i_1,\dots,i_{n-1})$. Now we give a formal definition of the metric entropy using the above notation.
\begin{defn}
The metric entropy of the mapping $T$ is defined as follows:
\begin{align*}
H(\alpha) &= H(\mu(A_1),\dots,\mu( A_k)) = -\sum_{i=0}^{n-1}\mu(A_i)\log \mu(A_i),\\
h_{\mu}(T,\alpha) &= \lim_{n\rightarrow\infty}\frac{1}{n}H(\bigvee_{i=0}^{n-1}T^{-i}\alpha),\\
h_{\mu}(T) &= \sup_{\alpha}h_{\mu}(T,\alpha).
\end{align*}
\end{defn}
Moreover, for ergodic measures we have the following characterization of the metric entropy presented by the authors of \cite{BK}:
\begin{thm}
Let $\mu$ be an ergodic measure. Then the following is true $\mu$-a.e. 
$$
h_{\mu}(T) = \lim_{\eps\rightarrow 0} \limsup_{n\rightarrow\infty}-\frac{1}{n}\log \mu(B_n(x,\eps)) = \lim_{\eps\rightarrow 0}\liminf_{n\rightarrow\infty}-\frac{1}{n}\log \mu(B_n(x,\eps)).
$$
\end{thm}

\subsection{Topological entropy}
Below we will  define two types of topological entropy, namely the upper capacity topological entropy and Bowen topological entropy. The reader should keep in mind that values of these entropies are equal for invariant and compact spaces.

Let $E\subseteq X$. A set $S\subseteq X$ is \textit{$(n,\eps)$-separated} for $E$ if $S\subseteq E$ and  $d_{n}(x,y)>\eps$ for any $x,y \in S$, $x\neq y$. A set $S\subseteq X$ is \textit{$(n,\eps)$-spanning} for $E$ if $S\subseteq E$ and for any $x \in X$ there exists $y \in S$ such that $d_{n}(x,y)<\eps$.
Define:
\begin{eqnarray*}
s_n(E,\eps) &=& \sup\{|S|: S \text{ is }(n,\eps)\text{-separated for } E\},\\
r_n(E,\eps) &=& \inf\{|S|:S \text{ is }(n,\eps)\text{-spanning for }E\}.
\end{eqnarray*}
It is true that:
\begin{equation}\label{SepSpanSet}
	r_n(E,\eps)\leq s_n(E,\eps)\leq r_n(E,\frac{\eps}{2}).
\end{equation}
\begin{defn}
The \textit{upper capacity topological entropy} of $E\subset X$ is defined by the following formula:
	$$
	h_d(T,E) = \lim_{\eps\rightarrow 0}\limsup_{n\rightarrow\infty}\frac{\log s_n(E,\eps)}{n}=\lim_{\eps\rightarrow0}\limsup_{n\rightarrow\infty}\frac{\log r_n(E,\eps)}{n}.
	$$
\end{defn}
Now let $\mathcal{G}_n(E,\eps)$ be the collection of all finite or countable coverings of the set $E$ with the Bowen balls $B_v(x,\eps)$ for $v\geq n$. Define:
$$
C(E;t,n,\eps,T) = \inf_{C \in \mathcal{G}_n(E,\eps)}\sum_{B_v(x,\eps) \in C}e^{-tv}
$$
and
$$
C(E;t,\eps,T) = \lim_{n\rightarrow\infty}C(E;t,n,\eps,T).
$$
Set:
$$
\htop(E,\eps,T) = \inf\{t: C(E;t,\eps,T)=0 \} = \sup\{t: C(E;t,\eps,T)=\infty \}.
$$
\begin{defn}
The \textit{Bowen topological entropy} of the set $E\subset X$ is defined by the following formula:
$$
\htop(T,E) = \lim_{\eps\rightarrow 0}\htop(E,\eps,T).
$$	
\end{defn}

By the definitions of the measure and the topological entropy we have that $h_{\mu}(T)\leq \htop(T,X)$. The following theorem known as a \emph{variational principle}  states a stronger relation between the metric and topological entropy:
\begin{thm}
Let $T:X\rightarrow X$ be a continuous map of a compact metric space $X$. Then:
$$
\htop(T,X) = \sup_{\mu}h_{\mu}(T),
$$
where the supremum is taken over all $T$-invariant Borel probability measures $\mu$.
\end{thm}
For ergodic measures we also have the \emph{Katok entropy formula} proved in \cite{K}, which is a generalization of Bowen's formula:
\begin{equation}\label{KatokEntrForm}
h_{\mu}(T) = \lim_{\eps\rightarrow 0}\lim_{n\rightarrow\infty}\frac{1}{n}N_T(n,\eps,\delta),
\end{equation}
where $N_T(n,\eps,\delta)$ denotes the smallest number of $(n,\eps)$-Bowen balls covering a~subset in $X$ of $\mu$-measure at least $1-\delta$ for some ergodic measure $\mu$.

\section{Main results}

\subsection{$\Phi$-irregular points and shadowing}
\begin{thm}\label{thmhtop}
Let $(X,T)$ be a dynamical system with the shadowing property and let $Y\subseteq X$ be a chain recurrent class.
If $\Phi\in \mathcal{C}(X,\mathbb{R})$ is such that there exist $\mu_1,\mu_2 \in \mathcal{M}_e(Y)$ with:
$$
\int \Phi d\mu_1 \neq \int \Phi d\mu_2
$$
then $\htop(T,I_\Phi(T))\geq \htop(T,Y)$.
\end{thm}

\begin{proof}
It suffices to show that for any $\gamma>0$ we have:
$$\htop(T,I_\Phi(T))\geq \htop(T,Y)-6\gamma.$$
We will achieve this goal by constructing a closed set $A\subset I_\Phi(T)$ such that:
$$\htop(T,A)\geq \htop(T,Y)-6\gamma.$$
So let $\gamma >0$ be fixed. Without loss of generality we can assume by the variational principle that:
\begin{equation}\label{VarPrinc}
h_{\mu_1}(T)>\htop(T,Y)-\gamma.
\end{equation}
Denote:
\begin{align}
\int\Phi d\mu_1&=\alpha, \label{alpha} \\
\int\Phi d\mu_2&=\beta. \label{beta}
\end{align}
Without loss of generality we may assume that $\alpha<\beta$. For sufficiently large $\xi_0~\in~(0,1)$ we have:
\begin{equation}\label{1-xi}
1-\xi_0 < \frac{\gamma}{\htop(T,Y)}.
\end{equation}
Take $\rho \in (1/2,1)$ such that:
\begin{equation}\label{rho}
1-\xi_0\rho < \frac{\gamma}{\htop(T,Y)}.
\end{equation}
Choose $\eta>0$ such that $8\eta< (1-\xi_0)(\beta-\alpha)$. The fuction $\Phi$ is continuous, so it is also bounded on $X$. In particular there is $M>0$ such that for all $x \in X$ we have $|\Phi(T^ix)|+\beta<M$.

Since $\mu_1$-generic points have full $\mu_1$ measure, for sufficiently large $\tilde{L}_1>0$ there is a subset $D_1\subseteq Y$ with $\mu_1(D_1)>\frac{3}{4}$ such that for every $x \in D_1$ and for every $n>\tilde{L}_1$ we have:
\begin{equation}\label{LCond}
|\frac{1}{n}\sum_{i=0}^{n-1}\Phi(T^ix)-\alpha|<\frac{\eta}{4}.
\end{equation}
By the Katok entropy formula (\ref{KatokEntrForm}) for each sufficiently small $\eps>0$ there is $N>0$ such that for all $n>N$ we have:
$$
\frac{1}{n}\log N_T(n,4\eps,\frac{1}{2})>h_{\mu_1}(T)-\gamma,
$$
where $N_T(n,4\eps,\frac{1}{2})$ is defined like in (\ref{KatokEntrForm}). If we denote by $E_m$ the maximal $(m,4\eps)$-separated set in $D_1$ we have for $m > N$:
\begin{equation}\label{E'set}
|E_m|> N_T(m,4\eps,\frac{1}{2})> e^{m(\htop(T,Y)-\gamma)},
\end{equation}
since $\{B_m(x,4\eps)\}_{x \in E_m}$ is a cover of $D_1$. Moreover, by decreasing $\eps$ if necessary, we can assume that $d(x,y)<\eps$ implies $|\Phi(x)-\Phi(y)|<\eta$. Fix $\delta>0$ such that any $2\delta$-pseudo-orbit is $\eps$-traced. Let $\mathcal{U} = \{U_i\}_{i=1}^S$ be a finite open cover of $Y$ such that $\mesh (\mathcal{U}) < \delta $. 
For $i,j \in \{1,\dots,S\}$ by $\omega_{ij}$ we denote the length of a chosen $\delta$-chain between $U_i$ and $U_j$. Let $\omega_{max} = \max_{i,j \in \{1,\dots,S\}}\omega_{ij}$.
For every $m\geq N$ and $i,j~\in~\{1,\dots,S\}$ we define a family of sets:
$$
\Lambda^{(m)}_{ij}=\{x \in E_m\cap U_i:T^mx \in U_j\}.
$$
For each $m>N$ by $i_m,j_m$ we denote the index of the set with the highest cardinality, that is:
$$
|\Lambda^{(m)}_{i_mj_m}| = \max\left\{|\Lambda^{(m)}_{ij}|:i,j\in \{1,\dots,S\}\right\}.
$$
Note that there are exactly $S$ sets in the cover which gives us at most $S^2$ possible pairs $(i_m,j_m)\in \{1,\dots,S\}\times \{1,\dots,S\}$. Hence for the sequence $\{\Lambda^{(m)}_{i_mj_m}\}_{m=N+1}^{\infty}$ we can find an increasing subsequence $\{m_k\}_{k=0}^{\infty}$ and an indexing pair $(\bar{i},\bar{j})$ which occurs infinitely many times, i.e.  $(i_{m_k},j_{m_k}) = (\bar{i},\bar{j})$ for all $k\geq 0$.

Denote $U=U_{\bar{i}}$, $V = U_{\bar{j}}$ and $E'_{m_k} = \Lambda^{(m_k)}_{\bar{i}\bar{j}}$. That way we get a new family $\{E'_{m_k}\}_{k=0}^{\infty}$ of $(m_k,4\eps)$-separated sets, with:
$$
E'_{m_k} = \{ x \in E_{m_k}\cap U: T^{m_k}x \in V \}
$$
and clearly:
$$
S^2|E'_{m_k}|\geq |E_{m_k}| > e^{m_k(\htop(T,Y)-\gamma)}.
$$
Consequently for sufficiently large $k$ we have:
\begin{equation}\label{2gamma}
|E'_{m_k}|> e^{m_k(\htop(T,Y)-2\gamma)}.
\end{equation}
Choose $k_1 \in \mathbb{N}$ large enough so that:
\begin{equation}\label{LBound}
L=m_{k_1} > \frac{\omega_{max}M}{\eta}
\end{equation}
and each $m_k$ for $k>k_1$ satisfies both (\ref{LCond}) and (\ref{2gamma}).

Analogously, since $\mu_2$-generic points have full $\mu_2$ measure, for sufficiently large $\tilde{L}_2>0$ there is a subset $D_2\subseteq Y$ with $\mu_2(D_2)>3/4$ such that for every $x \in D_2$ and for every $n>\tilde{L}_2$ we have:
\begin{equation}\label{L'Cond}
|\frac{1}{n}\sum_{i=0}^{n-1}\Phi(T^ix)-\beta|<\frac{\eta}{4}.
\end{equation}
Take:
$$
\tilde{L} = \max\{\tilde{L}_1,\tilde{L}_2\}.
$$
Choose some sufficiently large $k_2 \in \mathbb{N}$ such that:
$$
m_{k_2} > \max \{ \frac{2\tilde{L}}{\xi_0}, \frac{\tilde{L}}{1-\xi_0} \}
$$
and there is $J>m_{k_2}$ such that:
$$
\frac{\xi_0\rho}{2} < \frac{m_{k_2}}{J} < \xi_0
$$
and
\begin{equation}\label{KBound}
J>\frac{M\omega_{max}}{\eta}.
\end{equation}
It follows that:
\begin{align}
8\eta&< (1-\frac{m_{k_2}}{J})(\beta-\alpha),\\
(1-\frac{m_{k_2}}{J})&<\frac{\gamma}{\htop(T,Y)}.\label{gammaentr}
\end{align}
Define $\xi = \frac{m_{k_2}}{J}$. That way $\xi J$ satisfies (\ref{LCond}) and (\ref{2gamma}), $(1-\xi) J$ satisfies (\ref{L'Cond}) and both $\xi J$ and $(1-\xi)J$ are integers. Let:
 $$
 \zeta = \xi\alpha + (1-\xi)\beta.
 $$
Fix an arbitrary $\mu_2$-generic point $y \in D_2$. Let $U'$ be the set from $\mathcal{U}$ such that $y \in U'$ and fix any $V' \in \mathcal{U}$ such that $T^{(1-\xi)J}(y) \in V'$. By the definitions of $\omega_{max}$ there are $\delta$-chains $\{p_i\}_{i=0}^{P}$, $\{q_i\}_{i=0}^{Q}$, $\{w_i\}_{i=0}^{W}$ of length at most $\omega_{max}$ such that:
\begin{itemize}
	\item[(p)] $\{p_i\}_{i=0}^{P}$ is such that $p_0 \in V'$ and $p_P \in U$,
	\item[(q)] $\{q_i\}_{i=0}^{Q}$ is such that $q_0 \in V$ with $q_Q \in U$,
	\item[(w)] $\{w_i\}_{i=0}^{W}$ is such that $w_0 \in V$ with $w_W \in U'$.
\end{itemize}
In case when some of relevant sets are equal we simply assume that the corresponding $\delta$-chain is of length zero. If we put $K = J+W$, then $K$ satisfies (\ref{KBound}) as well. Take $\Gamma_K$ as the set of all $\delta$-pseudo-orbits of length $K$ built from the following blocks:
\begin{enumerate}
	\item block of length $\xi J$ of the orbit of some point from $E'_{\xi J}$ (by definition that point is $\mu_1$-generic),
	\item $\delta$-chain $\{w_i\}_{i=0}^{W-1}$ from $V$ to $U'$ (note that we skipped the last point of the chain),
	\item block of length $(1-\xi)J$ of the orbit of the chosen point $y \in U'$ (by definition $y$ is $\mu_2$-generic).
\end{enumerate}
Note that for any element $y \in \Gamma_K$ we have the following estimation:
\begin{multline}\label{zetasplit}
\left|\frac{1}{K}\sum_{i=0}^{K-1}\Phi(T^iy)-\zeta\right|   \leq \left|\frac{1}{K}\sum_{i=0}^{\xi J -1}\Phi(T^iy)-\frac{\xi J}{K}\alpha\right|\\ +
\left|\frac{1}{K}\sum_{i=\xi J}^{\xi J + W-1}\Phi(T^iy)-\frac{K-J}{K}\zeta\right| \\
+ \left|\frac{1}{K}\sum_{i=\xi J + W}^{K-1}\Phi(T^iy)-\frac{(1-\xi)J}{K}\beta\right|.
\end{multline}
Now let us present the estimates of the three summands above. For the first one we get:
\begin{equation}\label{zeta1}
\left|\frac{1}{K}\sum_{i=0}^{\xi J-1}\Phi(T^iy)-\frac{\xi J}{K}\alpha\right|  = \frac{\xi J}{K}\left|\frac{1}{\xi J}\sum_{i=0}^{\xi J-1}\Phi(T^iy)-\alpha\right|<\frac{\xi J }{K}\frac{\eta}{4}<\frac{\eta}{4},
\end{equation}
and, for the second one, by using (\ref{KBound}) and the definition of $K$, we get:
\begin{align}\label{zeta2}
\left|\frac{1}{K}\sum_{i=\xi J}^{\xi J +W-1}\Phi(T^iy) \right.&\left.- \frac{K-J}{K}\zeta\right|  < \frac{1}{K}\sum_{i=0}^{W-1}\left|\Phi(T^{i+\xi J}y)-\frac{K-J}{W}\zeta\right|\\
&  = \frac{1}{K}\sum_{i=0}^{W-1}\left|\Phi(T^{i+\xi J}y)-\zeta\right|  <\frac{WM}{K}< \eta. \nonumber
\end{align}
And, finally, for the third one we have:
\begin{align}\label{zeta3}
\left|\frac{1}{K}\sum_{i=\xi J +W}^{K-1}\Phi(T^iy)\right. &\left.-\frac{(1-\xi)J}{K}\beta\right|  = \frac{(1-\xi)J}{K}\left|\frac{1}{(1-\xi)J}\sum_{i=\xi J + W}^{K-1}\Phi(T^iy)-\beta\right|\\
& = \frac{(1-\xi) J}{K}\left|\frac{1}{(1-\xi)J}\sum_{i=0}^{(1-\xi)J-1}\Phi(T^{i+\xi J + W}y)-\beta\right| \nonumber \\
&<\frac{(1-\xi)J}{K}\frac{\eta}{4}<\frac{\eta}{4}. \nonumber
\end{align}
Going back to (\ref{zetasplit}), taking (\ref{zeta1}), (\ref{zeta2}), (\ref{zeta3}) into consideration we have:
\begin{equation}\label{KzetaEq}
\left|\frac{1}{K}\sum_{i=0}^{K-1}\Phi(T^iy)-\zeta\right| < 2\eta.
\end{equation}
By $(\ref{gammaentr})$ we have:
\begin{equation}\label{GammaK}
|\Gamma_K|\geq |E'_{\xi J}| \geq e^{\xi J(\htop(T,Y)-2\gamma)}\geq e^{J(\htop(T,Y)-3\gamma)}.
\end{equation}

{\bf Construction (*)}
 We are going to define a $\delta$-pseudo-orbit $\mathcal Z=\{z_i\}_{i=1}^{\infty}$ in $Y$ combining cyclically two types of blocks:
\begin{enumerate}\label{blockTypes}
\item[(C1)] blocks of length $L+Q$ consisting of the part of the orbit $x_{[0,L-1]}$ of some point $x$ from $E'_L$ and the $\delta$-chain $\{q_i\}_{i=0}^{Q-1}$ returning from $V$ to $U$ (note that $T^{L}x, q_0 \in V$ and $q_Q, x \in U$),
\item[(C2)] blocks of length $K+P$ consisting of the pseudo-orbit from $\Gamma_K$ and $\delta$-chain $\{p_i\}_{i=0}^{P-1}$ returning to $U$ (note that $T^{(1-\xi)J}y, p_0 \in V'$ and $p_P \in U$ and the first point of each pseudo-orbit in $\Gamma_K$ is in $U$).
\end{enumerate}
By the above any concatenation of sequences of $\delta$-pseudo-orbits (C1), (C2) is a~$\delta$-pseudo-orbit as well.

Before we start the construction of the $\delta$-pseudo-orbit we would like to make sure that the total length of the blocks (C1) and (C2) is divisible by the same number in every step of the construction. That is why we choose $\lambda, \kappa >0$ such that:
$$
\lambda(L+Q) = \kappa(K+P).
$$
Without loss of generality we may assume that $\lambda Q \geq \kappa P$ which implies:
\begin{equation}\label{LKrelation}
\lambda L\leq \kappa K.
\end{equation}
We also choose two increasing sequences of integers $\{l_n\}_{n=1}^{\infty}$, $\{l_n'\}_{n=1}^{\infty}$ and two inductively defined sequences $\{a_n\}_{n=1}^\infty$ and $\{b_n\}_{n=1}^\infty$ (see (\ref{EQ:anbn}) below) so that if we denote:
$$M_n = l_n\lambda (L+Q),$$
and
$$M_n' = l_n'\kappa (K+P),$$
then the following conditions are satisfied:
\begin{eqnarray}
M_n &>& \frac{M-\eta}{\eta}a_n, \label{MCond}\\
M_n' &>& \frac{M-\eta}{\eta}b_n\label{M'Cond}.
\end{eqnarray}
The sequences of integers $\{a_n\}_{n=1}^{\infty}$ and $\{ b_n\}_{n=1}^{\infty}$ are defined inductively as follows for $n \geq 1$:
\begin{align}
a_1 &= 0,\nonumber \\
b_n &= a_n+ M_n,  \label{EQ:anbn}\\
a_{n+1} &= b_n+M_n' = a_n + M_n + M_n'. \nonumber
\end{align}
Values $M_n$ (resp. $M'_n$) determine the length of the blocks $\tx^{(n)}$ (resp. $\ty^{(n)}$) which we will use in the construction of the $\delta$-pseudo-orbit below. The first one is the~concatenation of $l_n\lambda$ blocks (C1). 
The second is the concatenation of $l_n'\kappa$ blocks (C2). 
Strictly speaking:
\begin{align*}
\tx^{(n)} &= x^{(1)}_{[0,L-1]}q_0\dots q_{Q-1} x^{(2)}_{[0,L-1]}q_0\dots q_{Q-1}\dots x^{(l_n\lambda)}_{[0,L-1]}q_0\dots q_{Q-1}\\
\ty^{(n)}  &= y^{(1)}_{[0,K-1]}p_0\dots p_{P-1}y^{(2)}_{[0,K-1]} p_0\dots p_{P-1}\dots y^{(l_n'\kappa)}_{[0,K-1]}p_0\dots p_{P-1},
\end{align*}
for some not necessarily pairwise distinct points $x^{(i)} \in E'_L\subseteq U$ for $i=1,\dots,l_n\lambda$ and $y^{(j)} \in \Gamma_K$ for $j =1,\dots, l_n'\kappa$. Note that $\tx^{(n)}$, $\ty^{(n)}$ are in fact a whole family of $\delta$-pseudo-orbits depending on the choice of the blocks $x^{(i)}$ and elements of $\Gamma_K$ respectively.

Observe that the value $b_n$ tells us what is the length of the block of the pseudo-orbit $\mathcal{Z}$ since the beginning till the moment when we used the block $\tx^{(n)}$ in the $n$th step of our construction. The value $a_n$ tells us what is the total length of the $\delta$-pseudo-orbit $\mathcal{Z}$ after $(n-1)$ steps of the construction, which is also the moment when we used the block $\ty^{(n-1)}$ in the $(n-1)$st step of the construction.

Now let us more precisely present the steps of the construction of the infinite $\delta$-pseudo-orbit $\mathcal{Z}$. As described above we first have:
\begin{align*}
z_{[0,M_1-1]} &= \tx^{(1)},\\
z_{[M_1,M_1+M_1'-1]} &= \ty^{(1)}.
\end{align*}
and then in consecutive steps of the construction we repeat this procedure by putting:
\begin{align*}
z_{[a_n,b_n-1]} &= \tx^{(n)},\\
z_{[b_n, b_n+M_n'-1]}  = z_{[b_n,a_{n+1}-1]}&= \ty^{(n)}.
\end{align*}
The full $n$th step of the construction can be seen on the following scheme:
$$\dots
\xymatrix{
*[A]{ }\ar@{|--|}[rr]^{\tx_{[0,M_n]}}_<{a_n}&\hspace{0,2cm}&*[B]{ }
\ar@{~|}[rr]^{\ty_{[0,M'_n]}}_<{b_n}&\hspace{0,5cm}&*[D]{ }
\ar@{ }[]_<{a_{n+1}}
}\dots
$$
This finishes the construction of a single pseudo-orbit $\mathcal Z$ (Construction (*)). Note that different choices for $\tx^{(n)}$ lead to many different elements $\mathcal{Z}$. Later we will calculate how it reflects to the value of entropy. \medskip

Now let set $A\subseteq X$ be the closure of the set of elements of all possible orbits $\eps$-tracing all possible $\delta$-pseudo-orbits $\mathcal{Z}$ achieved by Construction (*). It means that for $u \in A$ we will find a $\delta$-pseudo-orbit $\mathcal{Z} = \{z_i\}_{i=0}^{\infty}$ such that $d(T^i(u),z_i)<\eps$. To show that $A$ is indeed a subset of the $\Phi$-irregular set $I_{\Phi}(T)$ we need to show that the Birkhoff average diverges for every point $u \in A$, that is:
\begin{equation}\label{BirkhDivCond}
\liminf_{n\rightarrow\infty}\frac{1}{n}\sum_{i=0}^{n-1}\Phi(T^iu) \neq \limsup_{n\rightarrow\infty}\frac{1}{n}\sum_{i=0}^{n-1}\Phi(T^iu).
\end{equation}
Note that by the construction of the pseudo-orbit $\mathcal{Z}$ and the choice of the elements in the set $A$, blocks of the trajectory of every point $u \in A$ is within the $\eps$-distance from the trajectories of length $L$ of points $x^{(i)}$ whose orbits build the blocks appropriate $\tx^{(n)}$.

First we estimate the lower limit of the Birkhoff average for a point $u \in A$ using the subsequence $\{b_n \}_{n=1}^{\infty}$. In fact all calculations below are shown for arbitrarily chosen point $u$ being the element of orbit $\eps$-tracing some $\delta$-pseudo-orbit from Construction (*). However if we take any point $v \in A$ from the closure, then $v$ is the limit of some sequence of points from $A$ and so in the sufficiently long prefix of the orbit of $v$ we will find the same structure of blocks as in the elements of the sequence converging to $v$. That means the estimations below are true for such points as well.
By (\ref{MCond}) we have that:
$$
\frac{a_n}{b_n}M < \eta
$$
and by (\ref{LBound}) we have:
$$
\frac{l_n\lambda QM}{b_n} < \frac{l_n\lambda \omega_{max} M }{l_n \lambda L} \eta.
$$
Note that by the construction of the block $\tx^{(n)}$ we have for every $n \geq 1$:
\begin{multline}\label{bnaverage}
\frac{1}{b_n}\sum_{i=a_n}^{b_n-1}\Phi(T^iu) = \frac{1}{b_n}\sum_{j=1}^{l_n\lambda}\sum_{i=a_n+(j-1)(L+Q)}^{a_n+jL+(j-1)Q-1}\Phi(T^iu) \\
+ \frac{1}{b_n}\sum_{j=1}^{l_n\lambda}\sum_{i=a_n+jL+(j-1)Q}^{a_n+j(L+Q)-1}\Phi(T^iu).
\end{multline}
Hence for every sufficiently large $n$:
\begin{align*}
|\frac{1}{b_n}&\sum_{i=1}^{b_n-1}\Phi(T^iu)-\alpha|  \leq \frac{1}{b_n}|\sum_{i=0}^{a_n-1}(\Phi(T^iu)-\alpha)| + |\frac{1}{b_n}\sum_{i=a_n}^{b_n-1}(\Phi(T^iu)-\alpha)|\\
&\leq \frac{a_n}{b_n}M + |\frac{1}{b_n}\sum_{j=1}^{l_n\lambda}\sum_{i=a_n+(j-1)(L+Q)}^{a_n+jL+(j-1)Q-1}(\Phi(T^iu)- \alpha)|\\
&\qquad + |\frac{1}{b_n}\sum_{j=1}^{l_n\lambda}\sum_{i=a_n+jL+(j-1)Q}^{a_n+j(L+Q)-1}(\Phi(T^iu)-\alpha)|\\
&\leq \frac{a_n}{b_n}M + \frac{1}{b_n}\sum_{j=1}^{l_n\lambda}\sum_{i=a_n+(j-1)(L+Q)}^{a_n+jL+(j-1)Q-1}|\Phi(T^iu)-\Phi(T^{i-a_n-(j-1)(L+Q)}x^{(j)})| \\
&\qquad + |\frac{1}{b_n}\sum_{j=1}^{l_n\lambda}\sum_{i=a_n+(j-1)(L+Q)}^{a_n+jL+(j-1)Q-1}(\Phi(T^{i-a_n-(j-1)(L+Q)}x^{(j)})-\alpha)|\\
&\qquad + \frac{1}{b_n}\sum_{j=1}^{l_n\lambda}\sum_{i=a_n+jL+(j-1)Q}^{a_n+j(L+Q)-1}|\Phi(T^iu)-\alpha|\\
&\leq \frac{a_n}{b_n}M + \frac{l_n\lambda L}{b_n}\eta + |\frac{1}{b_n}\sum_{j=1}^{l_n\lambda}\sum_{i=0}^{L-1}(\Phi(T^ix^{(j)})-\alpha)| + \frac{l_n\lambda Q}{b_n}M\\
& \leq 3\eta + |\frac{1}{b_n}\sum_{j=1}^{l_n\lambda}\sum_{i=0}^{L-1}(\Phi(T^ix^{(j)})-\alpha)|.
\end{align*}
Now observe that by the choice of $L$ the last component can be bounded from above as follows:
\begin{align*}
|\frac{1}{b_n}\sum_{j=1}^{l_n\lambda}\sum_{i=0}^{L-1}(\Phi(T^ix^{(j)})-\alpha)| &\leq \frac{L}{b_n}|\sum_{j=1}^{l_n\lambda}(\frac{1}{L}\sum_{i=0}^{L-1}\Phi(T^iu)-\alpha)|\\
&\leq \frac{l_n\lambda L\eta}{4b_n} \leq \frac{\eta}{4}.
\end{align*}
It follows that:
\begin{equation}\label{liminf}
|\frac{1}{b_n} \sum_{i=1}^{b_n-1}\Phi(T^iu)-\alpha| \leq 4\eta,
\end{equation}
which means that
$$
\liminf_{n\rightarrow\infty}\frac{1}{n}\sum_{i=0}^n\Phi(T^iu) \leq \alpha + 4\eta.
$$
To estimate the upper limit of the Birkhoff average of the point $u \in A$ we use the subsequence $\{ a_n \}_{n=1}^{\infty}$.
By (\ref{M'Cond}) we also have that:
$$
\frac{b_n}{a_{n+1}}M < \eta,
$$
and similarly to (\ref{bnaverage}) we will consider the following splitting for every $n\geq 1$:
\begin{align*}
\frac{1}{a_{n+1}}\sum_{i=b_n}^{a_{n+1}-1}\Phi(T^iu) &= \frac{1}{a_{n+1}}\sum_{j=1}^{l'_n\kappa}\sum_{i=b_n+(j-1)(K+P)}^{b_n+jK+(j-1)P-1}\Phi(T^iu)\\
&+ \frac{1}{a_{n+1}}\sum_{j=1}^{l'_n\kappa}\sum_{i=b_n+jK+(j-1)P}^{b_n+j(K+P)-1}\Phi(T^iu).
\end{align*}
Hence for every suficciently large $n$ we have:

\begin{align*}
|\frac{1}{a_{n+1}} (\Phi(T^iu)-\zeta)|&\leq \frac{1}{a_{n+1}}|\sum_{i=0}^{b_n-1}(\Phi(T^iu)-\zeta)| + \frac{1}{a_{n+1}}|\sum_{i=b_n}^{a_{n+1}-1}(\Phi(T^iu)-\zeta)|\\
& \leq \frac{b_n}{a_{n+1}}M + |\frac{1}{a_{n+1}}\sum_{j=1}^{l'_n\kappa}\sum_{i=b_n+(j-1)(K+P)}^{b_n+jK+(j-1)P-1}(\Phi(T^iu)-\zeta)| \\
&\qquad + |\frac{1}{a_{n+1}}\sum_{j=1}^{l'_n\kappa}\sum_{i=b_n+jK+(j-1)P}^{b_n+j(K+P)-1}(\Phi(T^iu)-\zeta)|\\
&\leq \frac{b_n}{a_{n+1}}M \\
&+ \frac{1}{a_{n+1}}\sum_{j=1}^{l'_n\kappa}\sum_{i=b_n+(j-1)(K+P)}^{b_n+jK+(j-1)P-1}|\Phi(T^iu) - \Phi(T^{i-b_n-(j-1)(K+P)}y^{(j)})|\\
&\qquad+ |\frac{1}{a_{n+1}}\sum_{j=1}^{l'_n\kappa}\sum_{i=b_n+(j-1)(K+P)}^{b_n+jK+(j-1)P-1}(\Phi(T^{i-b_n-(j-1)(K+P)}y^{(j)} - \zeta)| \\
&\qquad+ \frac{1}{a_{n+1}}\sum_{j=1}^{l'_n\kappa}\sum_{i=b_n+jK+(j-1)P}^{b_n+j(K+P)-1}|\Phi(T^iu)-\zeta|\\
&\leq \frac{b_n}{a_{n+1}}M + \frac{l'_n\kappa K}{a_{n+1}}\eta + |\frac{1}{a_{n+1}}\sum_{j=1}^{l'_n\kappa}\sum_{i=0}^{K-1}(\Phi(T^iy^{(j)})-\zeta)| + \frac{l'_n\kappa P}{a_{n+1}}M\\
&\leq 3\eta + |\frac{1}{a_{n+1}}\sum_{j=1}^{l'_n\kappa}\sum_{i=0}^{K-1}(\Phi(T^iy^{(j)})-\zeta)|.
\end{align*}

The last component is bounded as follows:
\begin{align*}
|\frac{1}{a_{n+1}}\sum_{j=1}^{l_n'\kappa}\sum_{i=0}^{K-1}\Phi(T^iy^{(j)})-\zeta|&\leq \frac{K}{a_{n+1}}|\sum_{j=1}^{l_n'\kappa}(\frac{1}{K}\sum_{i=0}^{K-1}\Phi(T^iu)-\zeta)|\\
&\leq \frac{l_n'\kappa K\eta}{4a_{n+1}}\leq \frac{\eta}{4}.
\end{align*}
It follows that:
\begin{equation}\label{limsup}
|\frac{1}{a_{n+1}} \sum_{i=0}^{a_{n+1}-1}\Phi(T^iu)-\zeta| \leq 4\eta.
\end{equation}
Altogether by (\ref{liminf}) and (\ref{limsup}) we have:
$$
\liminf_{n\rightarrow\infty}\frac{1}{n}\sum_{i=0}^{n-1}\Phi(T^iu) \leq \alpha + 4\eta < \zeta-4\eta \leq  \limsup_{n\rightarrow\infty}\frac{1}{n}\sum_{i=0}^{n-1}\Phi(T^iu),
$$
which means that the condition (\ref{BirkhDivCond}) is satisfied and the set $A$ is indeed a subset of the $\Phi$-irregular set $I_{\Phi}(T)$.

The final step of the proof is to estimate the value of topological entropy of the set $A$. Denote $h = h_{\mu_1}(T)-5\gamma$ and recall that by (\ref{VarPrinc}) we have $h  > \htop(T,Y)-6\gamma$. We are going to prove that:
$$
C(A;h,\eps,T)\geq 1,
$$
as it implies $\htop(T,A)\geq h$.
The constructed set $A$ is compact, so we can restrict our attention to a finite covers 
of $A$. We will show that for every sufficiently large $n$ we have:
$$
\sum_{B_v(x,\eps)\in \mathcal{C}}e^{-hv}\geq 1
$$
for every finite cover $\mathcal{C} \in \mathcal{G}_n(A,\eps)$. Of course the choice of $n$ affects the choice of the admissible coverings $\mathcal{C}$ of the set $A$ in the above sum.

Fix an integer $q$ such that:
\begin{equation}\label{qcond}
\frac{q+1}{q}\leq \frac{h_{\mu}(T)-3\gamma}{h_{\mu}(T)-5\gamma}
\end{equation}
and a sufficiently large $r_0$ such that for each $r\geq r_0$ we have:
\begin{equation}\label{1/qcond}
\frac{\lambda(L+Q)}{a_r}<\frac{1}{q}.
\end{equation}
By $\mathcal{A} = \{A_1,\dots A_t,\}$ denote the alphabet of cardinality $t = \min\{ (|E'_L|^l, (|\Gamma_K|^k \}$. Each symbol of that alphabet will uniquely represent blocks of the $\delta$-pseudo-orbit $\mathcal{Z}$ in the following way. The blocks in each $\mathcal{Z}$ are of the following possible types:
\begin{enumerate}
	\item[(D1)] blocks of length $\lambda(L+Q)$ consisting of the parts of the orbit of some points $x \in E'_L$ intertwined by $\delta$-chains from $V$ to $U$,
	\item[(D2)] blocks of length $\kappa(K+P)$ consisting of pseudo-orbits from $\Gamma_K$ intertwined by $\delta$-chains from $V'$ to $U$.
\end{enumerate}
Using the above splitting every $\delta$-pseudo-orbit uniquely defines an infinite word over the alphabet $\mathcal{A}$. We restrict the number of the blocks of types (D1), (D2) when necessary, so that there are exactly $t$ choices of each type.
Using that representation, given some $\delta$-pseudo-orbit $\mathcal{Z}$ denote by $S_m = |A_1\dots A_m|$ the total length of prefix built from $n$ blocks of the form (D1), (D2) in the splitting of $\mathcal{Z}$, whose code is $A_1,\dots,A_m$.
By the definition of $S_m$ there are $r>0$ and $0\leq s < 2l_r$ such that:
\begin{eqnarray}
S_m &=& a_r+s\lambda (L+Q),\label{def:ar}\\
S_{m+1} &=& S_m + \lambda(L+Q)\nonumber.
\end{eqnarray}

For every $x \in A$ we have a unique $\delta$-pseudo-orbit $\mathcal{Z} = \{z_i\}_{i=0}^{\infty}$ derived from the construction (*) such that $d(T^ix,z_i)\leq \eps$. This $\mathcal{Z}$ defines a unique sequence of symbols $A_1,A_2\dots$ and associated increasing sequence of integers $S_m$. Therefore every cover $\mathcal{C} \in \mathcal{G}_{n}(A,\eps)$ induces a cover $\mathcal{C}'$ where each ball $B_v(x,\eps) \in \mathcal{C}$ is replaced by $B_{S_m}(x,\eps)\in \mathcal{C}'$ such that $S_m\leq v < S_{m+1}$.
We may assume $a_r$ from \eqref{def:ar} satisfies \eqref{1/qcond} since we are interested only in large values of $n$, and this easily ensures $r>r_0$. Observe that by the definition:
\begin{equation}\label{BowenBallEst}
\sum_{B_v(x,\eps)\in \mathcal{C}}e^{-hv}\geq \sum_{B_{S_m}(x,\eps)\in \mathcal{C}'}e^{-hS_{m+1}}.
\end{equation}
Consider such a cover $\mathcal{C}'$ and let $c$ be the largest value $m$ used for replacement of $B_v(x,\eps)$ by $B_{S_m}(z,\eps) \in \mathcal{C}'$. 
By $\mathcal{W}_m$ define the set of all possible words of length $m$ over $\mathcal{A}$, while by $\mathcal{V}_c$ denote the set of all possible words of length at most $c$, that is:
$$
\mathcal{V}_c = \bigcup_{m=1}^c\mathcal{W}_m.
$$
Any point $p \in \mathcal{W}_m$ is a unique representation of some prefix of lengths $m$ of a $\delta$-pseudo-orbit. Furthermore, observe that each word of length $m$ over $\mathcal{A}$ is a prefix of exactly $\frac{|\mathcal{W}_c|}{|\mathcal{W}_m|}$ words from $\mathcal{W}_c$. Therefore, if we consider a set $\mathcal{W}\subseteq \mathcal{V}_c$ containing some prefixes of all words from $\mathcal{W}_c$, then:
$$
\sum_{m=1}^c |\mathcal{W}\cap \mathcal{W}_m|\frac{|\mathcal{W}_c|}{|\mathcal{W}_m|}\geq |\mathcal{W}_c|,
$$
as for every word in $\mathcal{W}_c$ one of its prefixes has to be in $\mathcal{W}\cap \mathcal{W}_m$ for some $1\leq m\leq c$ and the number of words in $\mathcal{W}_c$ with that prefix cannot exceed $\frac{|\mathcal{W}_c|}{|\mathcal{W}_m|}$. By the above discussion, if $\mathcal{W}$ contains a prefix of every word from $\mathcal{W}_c$ then:
$$
\sum_{m=1}^c \frac{|\mathcal{W}\cap\mathcal{W}_m|}{|\mathcal{W}_m|}\geq 1.
$$

Recall that any $x \in B_{S_m}(z,\eps)$ defines uniquely a word $A_{i_1},\dots, A_{i_m} \in\mathcal{A}$, $i_1,\dots,i_m \in \{1,\dots,t\}$ such that $x_{[0,S_m]} \approx A_{i_1}\dots A_{i_m} \in \mathcal{W}_c$, where $\approx$ denotes the identification of pseudo-orbit with symbols in $\mathcal{A}$. We claim that the block $x_{[0,S_m]}$ defines the same $m$-letter words as any $z \in B_{S_m}(x,\eps)\cap A$. Assume on the contrary that $z_{[0,S_m]}\approx A_{j_1}\dots A_{j_m}$ and $i_\iota\neq j_\iota$ for some $\iota$. Then in case of block of the form (C1) there are $p,q \in E'_L$, $p\neq q$ (the case of (C2) is analogous with $p,q \in E'_{\xi J}$) and $1\leq i\leq m $ and $0\leq \kappa<l$ such that:
$$
T^{S_{i-1}+\kappa(L+Q)}z \in B_L(p,\eps) \text{ and } T^{S_{i-1}+\kappa(L+Q)}x \in B_L(q,\eps)
$$
and so for $0\leq j< L$ we also have:
\begin{multline*}
d(T^jp,T^jq) \leq d(T^jp,T^{S_{i-1}+\kappa(L+Q)+j}z)\\ + d(T^{S_{i-1}+\kappa(L+Q)+j}z,T^{S_{i-1}+\kappa(L+Q)+j}x) \\+ d(T^{S_{i-1}+\kappa(L+Q)+j}x, T^jq)  <3\eps.
\end{multline*}
This is the contradiction, as the set $E'_L$ is $(L,4\eps)$-separated. Indeed the claim holds, that is each $B_{S_m}(z,\eps)$ defines a unique word over $\mathcal{A}$.
This immediately implies that:
\begin{equation}\label{WmBallEst}
\sum_{B_{S_m}(x,\eps)\in \mathcal{C}'}\frac{1}{|\mathcal{W}_m|}\geq 1.
\end{equation}
By (\ref{GammaK}) and $(\ref{LKrelation})$ we have the following estimation for the cardinality of set $\mathcal{W}_m$:
\begin{equation}\label{WmCard}
|\mathcal{W}_m| = t^m \geq e^{m\lambda L(\htop(T,Y)-3\gamma)}.
\end{equation}
Moreover by (\ref{qcond}) and (\ref{1/qcond}):
\begin{align*}\label{SEst}
\frac{S_{m+1}}{S_m} &= 1 + \frac{\lambda(L+Q)}{S_m} \leq 1+\frac{\lambda(L+Q)}{a_r}\\
&\leq \frac{q+1}{q} \leq \frac{h_{\mu}(T)-3\gamma}{h_{\mu}(T)-5\gamma}.
\end{align*}
Combining (\ref{BowenBallEst}), (\ref{WmBallEst}), (\ref{WmCard}) we get:
\begin{align*}
\sum_{B_v(x,\eps) \in \mathcal{C}}e^{-hv}&\geq \sum_{B_{S_m}(x,\eps) \in \mathcal{C}'}e^{-S_m(\htop(T,Y)-3\gamma)}\\
&\geq \sum_{B_{S_m}(x,\eps) \in \mathcal{C}'}e^{-m\lambda L(\htop(T,Y)-3\gamma)} \geq 1.
\end{align*}
That implies:
$$
\htop(T,A)\geq h_{\mu_1}(T)-5\gamma > \htop(T,Y)-6\gamma.
$$
The proof is finished.
\end{proof}

\begin{cor}\label{cor:CRclass}
	Let $(X,T)$ be a dynamical system with the shadowing property and let $Y\subset X$ be a chain recurrent class. If $\Phi \in \mathcal{C}(X,\mathbb{R})$ is such that $I_\Phi(T)\cap Y \neq \emptyset$ then $\htop(T,I_\Phi(T))\geq \htop(T,Y)$.
\end{cor}
\begin{proof}
	Since $I_\Phi(T)\cap Y\neq \emptyset$, there is a point $x\in Y$ such that $\lim_{n\to\infty} \frac{1}{n}\sum_{i=0}^{n-1}\Phi(T^ix)$ does not exist. By ergodic decomposition \cite[Lemma~2.1]{Thom} we have that there are ergodic measures $\mu_1, \mu_2 \in \mathcal{M}_e(Y)$ such that $\int\Phi d\mu_1\neq \int\Phi d\mu_2$. The result follows by Theorem \ref{thmhtop}.
\end{proof}

By Corollary~\ref{cor:CRclass} we obtain that:
$$
\htop(T,I_{\Phi}(T)) \geq \sup \left\{\htop(T,Y): I_{\Phi}(T)\cap Y \neq \emptyset \text{ and }Y\in C_c(X,T) \right\}.
$$
where $C_c(X,T)$ denotes the set of all chain recurrent classes of $(X,T)$. 
Unfortunately, we were unable to prove that the inequality cannot be strict in the above formula.
However, if we slight change the above set, it ensures equality as shown below.
\begin{cor}\label{CorSup}
Let $(X,T)$ be a dynamical system with the shadowing property and let $\Phi:X\rightarrow \mathbb{R}$ be a continuous function. Then we have:
\begin{equation}\label{EQ:EXT}
\htop(T,I_{\Phi}(T)) = \sup\{\htop(T,Y):I_{\Phi}(T)\cap Y^{\omega}\neq\emptyset \text{ and }Y\in C_c(X,T) \}.
\end{equation}
\begin{proof}
Fix any chain-recurrent class $Y$ such that $I_{\Phi}(T)\cap Y^{\omega}\neq\emptyset$.
Take any $x~\in~Y^\omega$ and let $V_T(x)$ denote the weak$^*$ limit set of $\frac{1}{n}\sum_{i=0}^{n-1}\delta_{T^ix}$. Any measure $\mu~ \in~V_T(x)$ is supported on $\omega_T(x)$, so it is supported on $Y$ as well. Hence there are measures $\mu_1, \mu_2$ with supports in $Y$ such that $\int \Phi d\mu_1 \neq \int \Phi d\mu_2$. By ergodic decomposition (see \cite[Lemma 2.1]{Thom}) we may assume that $\mu_1, \mu_2 $ are ergodic and so we may apply Theorem \ref{thmhtop} which proves that $\htop(T,I_{\Phi}(T))\geq \htop(T,Y)$. This proves ``$\geq$" in \eqref{EQ:EXT}.

To prove the opposite inequality ``$\leq$" in \eqref{EQ:EXT} we start by showing that: $$\htop(T,Y^{\omega})=\htop(T,Y).$$
We need the result from \cite{Bow} which states that letting:
$$QR(t)=\{ x \in X: \text{ there exists } \mu \in V_T(x) \text{ with }h_{\mu}(T)\leq t \},$$
we have $\htop(T,QR(t))\leq t$. Taking $t = \sup_{\mu \in V_T(x), x \in Y^{\omega}}h_{\mu}$ we see that $Y^{\omega}~\subseteq~QR(t)$ which gives:
\begin{eqnarray}\label{htopeq}
h_{top}(T,Y^{\omega}) &\leq & \sup\{h_{\mu}:\mu \in V_T(x),\; x \in Y^{\omega}\}\\
&\leq& \sup\{h_{\mu}: \supp(\mu)\subseteq Y\} = h_{top}(T,Y).\nonumber
\end{eqnarray}
while the opposite inequality is obvious by inclusion $Y\subset Y^\omega$. This way we obtain that if we denote:
$$
\hat{t} = \sup\{\htop(T,Y): I_{\Phi}(T)\cap Y^{\omega}\neq \emptyset \text{ and }Y\in C_c(X,T)\}
$$
then:
$$
\hat{t} = \sup\{\htop(T,Y^\omega): I_{\Phi}(T)\cap Y^{\omega}\neq \emptyset \text{ and }Y\in C_c(X,T)\}.
$$

Finally observe that if $x\in I_{\Phi}(T)$ then there is a chain-recurrent class $Y$ such that $\omega_T(x)\subset Y$ (e.g. see \cite{HSZ}) and so $I_{\Phi}(T)\cap Y^\omega\neq \emptyset$.
Therefore
$$
I_{\Phi}(T) \subseteq \bigcup_{Y \in C_c(X,T), Y^{\omega}\cap I_{\Phi}(T)\neq \emptyset}Y^{\omega}\subseteq QR(\hat{t}).
$$
and the proof is completed by the above mentioned result from \cite{Bow}.
\end{proof}
\end{cor}

\subsection{Typical $\Phi$-irregular sets}

For any dynamical system  $(X,T)$, from \cite[Lemma 3.3]{Tian2017} the set
$$\mathcal C^*:=\{\Phi\in C^0(X,\mathbb{R})\,:\,I_\Phi(T) \text{ is  not empty}\}$$
is either empty or is an open and dense subset in the space of continuous functions.
Suppose that $(X,T)$ has the shadowing property and has positive topological entropy. By \cite{DOT} the set of irregular points $I(T)$ is not empty and carries full topological entropy, in particular  there exists some $\Phi$ with $I_{\Phi}(T)\neq \emptyset.$ Thus  $\mathcal C^*$  is an open and dense subset in the space of continuous functions. On one hand, if additionally $f$~is transitive, then by Theorem~\ref{thmhtop} we obtain that:
$$\mathcal C^*=\{\Phi\in C^0(X,\mathbb{R})\,: \,I_\Phi(T) \text{ is  not empty and carries full topological entropy }\}$$
so that the later set is an open and dense subset in the space of continuous functions. On the other hand, if the system is not transitive, one can have some $\Phi$ such that  $I_\Phi$ is  not empty but does not carry full topological entropy.
For example, suppose that the system $(X,T)$ is composed by two disjoint transitive subsystem $(X_i,f_i)$, $i=1,2$ with shadowing for  which $f_{1}$  has entropy larger than the one of $f_{2}$ and suppose that $f_{2}$ has at least two invariant measures, then one can take a continuous function $\Phi$~such that $\Phi |_{X_1}=0$ but
$$\inf_{\mu\in \mathcal{M}_{f_2}(X_2)} \int \Phi d\mu<\sup_{\mu\in \mathcal{M}_{f_2}(X_2)}\int \Phi d\mu.$$
In this case $I_{\Phi}(T)$ is not empty and carries entropy equal to the one of $f_2$ (not equal to the entropy of the whole system). However, we can show that  ``most'' functions still have the property that  $I_\Phi(T)$ is  not empty and carries full topological entropy.

\begin{thm}\label{thmtyp}
Suppose that $(X,T)$ has the shadowing property and positive topological entropy. Then:
$$
\mathcal{R}=\{\Phi\in C^0(X,\mathbb{R}): I_\Phi(T) \text{ is  not empty and carries full topological entropy }\}
$$
is a dense  $G_\delta$ subset   in $C^{0}(X,\mathbb{R})$. In fact, for any $\eps>0$:
\begin{multline*}
\mathcal{R}_{\eps}=\{\Phi\in C^0(X,\mathbb{R}): I_\Phi(T) \text{ is  not empty} \\ \text{and carries entropy larger than }h_{top}(T)-\eps\}
\end{multline*}
is an open and dense  subset in $C^{0}(X,\mathbb{R})$.
\end{thm}

\begin{proof}
Clearly $\mathcal R=\bigcap_n \mathcal R_{\frac1n}$ so by Baire theorem it is enough to prove the second statement.

Fix any $\eps>0.$  First we are going to show that $\mathcal{R}_{\eps}$ is dense.
By \cite[Corollary 3.5]{DOT} there are two ergodic measures $\mu$ and $\nu$ such that $h_{\mu}(T)>h_{top}(T)-\eps$ and $\mu\neq \nu,$ $\supp(\nu)\subseteq \supp(\mu)$.
By the definition of weak$^*$ topology, there exists some continuous function $\Phi_\eps$ such that $\int\Phi_\eps d\mu\neq \int \Phi_\eps d\nu.$
Note that by ergodicity $T$ restricted to $\supp(\mu)$ is transitive, in particular $\supp(\mu)$ is contained in a chain recurrent class $Y$. Therefore by  Theorem~\ref{thmhtop} we see that
$$
h_{top}(I_{\Phi_\eps}(T))\geq h_{top}(T|_Y)\geq h_{\mu}(T)>h_{top}(T)-\eps.
$$

For any given continuous function $\Phi$,  if  $\inf_{\mu\in \mathcal{M}_T(Y)}\int\Phi d\mu<\sup_{\mu\in \mathcal{M}_{T}(Y)}\int\Phi d\mu,$ then by Theorem~\ref{thmhtop} together with ergodic decomposition theorem we obtain that $\Phi\in \mathcal{R}_{\eps}. $

In the second case, that is when $\inf_{\mu\in \mathcal{M}_{T}(Y)}\int\Phi d\mu=\sup_{\mu\in \mathcal{M}_{T}(Y)}\int\Phi d\mu,$ consider functions $\Phi_{n}=\Phi+
\frac1n\Phi_{\epsilon}$  which converge to $\Phi$ and satisfy:
$$
\inf_{\mu\in \mathcal{M}_{T}(Y)}\int\Phi_n d\mu<\sup_{\mu\in \mathcal{M}_{T}(Y)}\int\Phi_n d\mu.
$$
By the previous case, we see that $\Phi_n\in \mathcal{R}_{\eps}$ for every $n$. Indeed   $\mathcal{R}_{\eps}$ is dense.

Now we will show that $\mathcal{R}_{\eps}$ is also open. Once again we will need
result of Bowen from \cite{Bow} which states that letting
$$QR(t)=\{ x \in X: \text{ there exists } \mu \in V_T(x) \text{ with }h_{\mu}(T)\leq t \},$$
we have $\htop(T,QR(t))\leq t$. Take any $\Phi\in \mathcal{R}_{\eps}. $
We claim that there exists a point $y\in I_{\Phi}(T)$   such that  every $\mu\in V_{T}(y)$  we have $h_{\mu}(T)>h_{top}(T)-\epsilon$.
Otherwise, $ I_{\Phi}(T) \subseteq QR(\htop(T)-\eps)$ and Bowen's result implies that $h_\top(I_{\Phi}(T))\leq \htop(T)-\eps$
which is in contradiction to $\Phi\in \mathcal{R}_{\eps}$. This proves the claim.
In other words, there is $y\in I_{\Phi}(T)$ and measures $\mu,\nu\in V_T(y)$ such that $ \int\Phi d\mu\neq \int \Phi d\nu.$
Recall that $\cup_{\mu\in V_T(y)} \supp(\mu)\subseteq \omega_T(y).$ It is well known (e.g. see \cite{HSZ}) that $\omega_T(y)$ is always a~subset of some chain recurrent class, say $Y$, and consequently $V_T(y)\subseteq \mathcal M_{T}(Y)$. Thus:
\begin{align*}
h_{top}(Y)&=\sup\{h_\mu(T)|\, \mu\in \mathcal M_{T}(Y)\}\\
&\geq \sup\{h_\mu(T)|\, \mu\in V_T(x)\} \geq h_{\mu}(T)>h_{top}(T)-\eps
\end{align*}
and  $\inf_{\mu\in \mathcal{M}_{T}(Y)}\Phi d\mu<\sup_{\mu\in \mathcal{M}_{T}(Y)}\Phi d\mu.$
If we take sufficiently small open neighborhood $\mathcal U$ of $\Phi$ in $C^0(X,\mathbb{R})$ with supremum metric, then for any $\Psi\in U$ we will also have $\inf_{\mu\in \mathcal{M}_{T}(Y)}\int\Psi d\mu<\sup_{\mu\in \mathcal{M}_{T}(Y)}\int\Psi d\mu$. Using once again Theorem~\ref{thmhtop} we see that $\Psi\in \mathcal{R}_{\epsilon}.$
This completes the proof.
\end{proof}

\subsection{Level sets with respect to reference measure and shadowing}

The considerations below have their motivation in \cite{Young}. We assume that $m \in \mathcal{M}$ is a~finite Borel measure on $X$ and think of it as of a reference measure.
Define:
$$
h_m(T,x) = \lim_{\eps\rightarrow 0}\limsup_{n\rightarrow\infty}-\frac{\log m(B_n(x,\eps))}{n}
$$
and
$$
h_m(T,\nu) = \nu - \esssup h_m(T,x),
$$
where :
$$\nu - \esssup h_m(T,x) = \inf\left\{ \alpha \in \mathbb{R}: \nu\left(\{x \in X:h_m(T,x)>\alpha\}\right)=0\right\}.$$
Note that the measure $m$ need not be even $T$-invariant, however if we take $m$ as an ergodic measure, then $h_{m}(T,x) = h_{m}(T)$ $m$-a.e. by \cite{BK}. In particular if $m=\nu$ and $\nu$ is ergodic then $h_m(T,\nu) = h_{\nu}(T)$.
Now define the set:
\begin{multline*}
\mathcal{V}^-=\{\xi \in C(X,\mathbb R): \text{ there exist an arbitrarily small } \eps>0 \text{ and } C=C(\eps)\\
\text{ such that, for all }x \in X \text{ and }n \geq 0: m(B_n(x,\eps))\geq C e^{-\sum_{i=0}^{n-1}\xi(T^ix)}\}.
\end{multline*}
For bad choice of the reference measure $m$ it may happen that $\mathcal{V}^-$ is empty. However in many cases there are natural candidates for $m$ which also ensures $\mathcal{V}^-\neq \emptyset$ (e.g. see Example \ref{exV}).

For $\Phi \in \mathcal{C}(X,\mathbb{R})$ and $E\subseteq \mathbb{R}$ put:
$$
\underline{R}(\Phi, E)=\liminf_{n\rightarrow\infty}\frac{1}{n}\log m\left(\{x \in X: \frac{1}{n}\sum_{i=0}^{n-1}\Phi(T^ix) \in E\}\right).
$$
In \cite{Young} the author proved the [Theorem 1] 
for dynamical systems with the specification property. Here we state an analogous theorem for dynamical systems with the shadowing property.
\begin{thm}\label{thm:v-}
Let $(X,T)$ be a dynamical system with the shadowing property such that $\htop(T,X)<\infty$. Then for every $\Phi \in \mathcal{C}(X,\mathbb{R})$, every $c \in \mathbb{R}$ and every $\xi \in \mathcal{V}^-$ we have:
\begin{multline*}
\underline{R}(\Phi,(c,\infty))\geq \sup \{h_{\nu}(T)-\int\xi d\nu: \nu \in \mathcal{M}_T(X), \int \Phi d\nu>c,\\ \nu \text{ is supported on some chain recurrent class } Y\subseteq X\}.
\end{multline*}
\end{thm}
\begin{proof}
Fix some $\xi \in \mathcal{V}^-$, $c \in \mathbb{R}$ and $\Phi \in \mathcal{C}(X,\mathbb{R})$. Let
$$D_n = \{ x \in X: \frac{1}{n}\sum_{i=0}^{n-1}\Phi(T^ix)>c \}$$
 and pick an arbitrary $\nu \in \mathcal{M}$ supported on some chain recurrent class $Y\subseteq X$ with $\int\Phi d\nu >c$. Fix an arbitrary $\gamma>0$ and put $\delta = \frac{1}{4}(\int\Phi d\nu -c)$. Observe that, in fact, we have $\int \Phi d\nu = c+4\delta$.

By  Theorem 4.3 from \cite{LO} there exists a sequence of ergodic measures $\{\nu_n\}_{n \in \mathbb{N}}$ supported on some (possibly different) chain recurrent classes in $X$ such that $
\lim_{n\rightarrow\infty}\nu_n = \nu$ and $\lim_{n\rightarrow\infty}h_{\nu_n}(T)=h_{\nu}(T).$
Hence there exist $N_1, N_2>0$ such that:
\begin{align*}
\left|\int\Phi d\nu_n - \int \Phi d\nu\right| &< \delta \text{ and } \left|h_{\nu_n}(T)-h_{\nu}(T)\right|<\delta \text{ for all }n>N_1,\\
\left|\int \xi d\nu_n - \int \xi d\nu\right|& <\gamma \text{ and } |h_{\nu_n}(T)-h_{\nu}(T)|<\gamma \text{ for all }n>N_2.
\end{align*}
Put $M = \max\{N_1,N_2\}$ and denote
$$
\mu=\nu_M.
$$
Clearly, it follows that $\int \Phi d\mu > c+3\delta$ and $h_{\mu}(T) - \int \xi d\mu \geq h_{\nu}(T) - \int\xi d\nu - 2\gamma$.

Now we are going to prove that:
$$
\liminf_{n\rightarrow\infty}\frac{1}{n}\log m(D_n)\geq h_{\nu}(T) - \int \xi d\nu - 3\gamma,
$$
which requires estimating the value of $m(D_n)$ in terms of $h_{\mu}(T) - \int \xi d\mu$.
Choose $\eps>0$ such that $d(x,y)<\eps$ implies $|\Phi(x)-\Phi(y)|<\delta$.
If we denote by $N_T(n,2\eps,\frac{1}{2})$ the minimal number of $(n,2\eps)$-Bowen balls covering a set of $\mu$-measure at least $\frac{1}{2}$ 
and then by the Brin-Katok entropy formula we have (decreasing $\eps$ when necessary):
\begin{equation}
h_{\mu}(T)-\gamma/2 \leq \liminf_{n\rightarrow\infty}\frac{1}{n}\log N_T(n,2\eps,\frac{1}{2}).\label{eq:brin-katok}
\end{equation}

For $\xi \in \mathcal{V}^-$ chosen above and every $x\in X$ we have:
\begin{equation}\label{BBallMeasureEst}
\frac{1}{n}\log m(B_n(x,\eps))\geq \frac{1}{n}\log C - \frac{1}{n}\sum_{i=0}^{n-1}\xi(T^ix).
\end{equation}
As $\mu$ is ergodic, we have $\lim_{n\to\infty}\frac{1}{n}\sum_{i=0}^{n-1}\Phi(T^ix)>c+3\delta$ $\mu$-a.e.
Furthermore, note that if we take $x \in X$ such that $|\frac{1}{n}\sum_{i=0}^{n-1}\xi(T^ix)-\int\xi d\mu|<\gamma$ for all $n \geq N$ then, by (\ref{BBallMeasureEst}), we have:
\begin{equation}
\frac{1}{n}\log m(B_n(x,\eps))\geq -\frac{1}{n}\sum_{i=0}^{n-1}\xi(T^ix)>-\int \xi d\mu - \gamma.
\end{equation}
Again by ergodic theorem, $\mu$-a.e. point satisfies this condition for sufficiently large $N$.
In particular, there is an $N \in \mathbb{N}$ such that $\mu(D)>\frac{2}{3}$, where:
\begin{multline*}
D = \left\{ x \in X: 
|\frac{1}{n}\sum_{i=0}^{n-1}\xi(T^ix)-\int\xi d\mu|<\gamma \text{ and } \right. \\
\left. \frac{1}{n}\sum_{i=0}^{n-1}\Phi(T^ix)>c+\delta  \text{ for all } n \geq N \right\}.
\end{multline*}
Clearly $D\subseteq D_n$ for all $n\geq N$ and so $\mu(D_n)>\frac{2}{3}$. For each $x\in D$ and $y \in B_n(x,\eps)$ we have:
\begin{align*}
\frac{1}{n}\sum_{i=0}^{n-1}\Phi(T^iy)>-\delta + \frac{1}{n}\sum_{i=0}^{n-1}\Phi(T^ix)>c,
\end{align*}
therefore for each $n\geq N$ we have $B_n(x,\eps)\subseteq D_n$, provided that $x\in D$.

 Let $\mathcal{E}_n \subseteq D$ be the maximal $(n,2\eps)$-separated subset of $D$. Then for any distinct  $x, x' \in \mathcal{E}_n$
 we have $B_n(x,\eps)\cap B_n(x',\eps)=\emptyset$ and clearly $D\subset \bigcup_{x\in \mathcal{E}_n}B_{n}(x,2\eps)$, so:
 $$
 N_T(n,2\eps,\frac{1}{2})\leq |\mathcal{E}_n|.
 $$

 Thus, increasing the value of $N$ if necessary, by \eqref{eq:brin-katok} for $n\geq N$ we have:
 $$
 h_{\mu}(T)-\gamma \leq \frac{1}{n}\log |\mathcal{E}_n| 
 $$
 which we can equivalently write as:
 $$
 e^{n(h_{\mu}(T)-\gamma)}\leq |\mathcal{E}_n|.
 $$
Combining the above observations together, we see that for $n\geq N$:
 \begin{align*}
\frac{1}{n}\log m(D_n) &\geq  \frac{1}{n}\log \sum_{x \in \mathcal{E}_n}m(B_n(x,\eps)) \geq \frac{1}{n}\log \sum_{x \in \mathcal{E}_n}Ce^{-\sum_{i=0}^{n-1}\xi(T^ix)} \\
&\geq \frac{1}{n}\log C|\mathcal{E}_n|e^{-n(\int \xi d\mu +2\gamma)}
\geq \frac{1}{n}\log C e^{n(h_{\mu}(T)-\gamma)}e^{-n(\int\xi d\mu + 2\gamma)}.
 \end{align*}
 It follows that:
 \begin{align*}
 \liminf_{n\rightarrow\infty}\frac{1}{n}\log m(D_n)\geq h_{\mu}(T)-\int\xi d\mu -3\gamma
 \end{align*}
 which completes the proof.
\end{proof}

Now for $a\in \R$ and $\theta>0$ define the following two level sets:
\begin{align*}
R_{\Phi}(a) &= \left\{x \in X: \lim_{n\rightarrow\infty}\frac{1}{n}\sum_{i=0}^{n-1}\Phi(T^ix)=a\right\},\\
R_{\Phi}(a,\theta) &= \left\{x \in X: a-\theta<\liminf_{n\rightarrow\infty}\frac{1}{n}\sum_{i=0}^{n-1}\Phi(T^ix)\leq
\limsup_{n\rightarrow\infty}\frac{1}{n}\sum_{i=0}^{n-1}\Phi(T^ix)<a+\theta\right\}.
\end{align*}
By the definition $R_{\Phi}(a) = \bigcap_{\theta>0}R_{\Phi}(a,\theta)$, hence the set $R_{\Phi}(a)$ is much harder to deal with.

\begin{thm}\label{LevelSetThm}
Let $(X,T)$ be a dynamical system with the shadowing property and $\Phi \in \mathcal{C}(X,\mathbb{R})$, $a\in \R$ and $\theta>0$ be such that
$R_{\Phi}(a,\theta)\neq \emptyset$. Then
\begin{align*}
\htop(T,R_{\Phi}(a,\theta)) &= \sup \{ h_{\mu}(T):\int \Phi d\mu \in (a-\theta,a+\theta), \mu \in \mathcal{M}_e(X) \} \\
&= \sup \{h_{\mu}(T): \int \Phi d\mu \in (a-\theta,a+\theta), \supp(\mu)\subset Y \in  C_c(X,T) \}.
\end{align*}
\end{thm}

\begin{proof}
Given $x\in X$  as usual we denote by $V_T(x)$ the weak$^*$ limit set of $\frac{1}{n}\sum_{i=0}^{n-1}\delta_{T^ix}$. Recall that $\bigcup_{\mu \in V_T(x)}\supp(\mu)\subseteq \omega_T(x)$ so each $\mu\in V_T(x)$ is supported on some chain-recurrent class. Denote the set of generic points for $\mu$ by
$$
G_{\mu} = \{x \in X: V_T(x) = \{\mu\} \}.
$$
By the result of Bowen \cite[Theorem 3]{Bow} we have that $\htop(T,G_{\mu}) = h_{\mu}(T)$ provided that $\mu$ is ergodic. Note that if $\int \Phi d\mu \in (a-\theta,a+\theta)$ and $\mu$ is ergodic then the set of generic points for $\mu$ is a subset of $R_{\Phi}(a,\theta)$, which automatically implies:
$$
h_{\mu}(T)=\htop(T,G_{\mu})\leq \htop(T,R_{\Phi}(a,\theta))
$$
In other words
$$
\htop(T,R_{\Phi}(a,\theta)) \geq \sup \{ h_{\mu}(T):\int \Phi d\mu \in (a-\theta,a+\theta), \mu \in \mathcal{M}_e(X) \}.
$$
Recall that by the Bowen's result  from \cite[Theorem 2]{Bow} we know that: $$\htop(T,QR(t))~\leq~t,$$ where:
$$
QR(t) = \{ x \in X: \text{ there exists } \mu \in V_T(x) \text{ with }h_{\mu}(T)\leq t \}.$$
If we put
$$
\hat{t}=\sup_{x\in R_{\Phi}(a,\theta)}\sup_{\mu \in V_T(x)} h_\mu(T).
$$
then $R_{\Phi}(a,\theta)\subset QR(\hat{t})$ and so we have $\htop(T,R_{\Phi}(a,\theta))\leq \hat{t}$.
But for every $x\in X$ and $\mu \in V_T(x)$, by  \cite[Theorem~4.3]{LO} we can find a sequence of ergodic measures $\nu_n$ with $\nu_n\to \mu$ and $h_{\nu_n}(T)~\to~ h_\mu(T)$. By the definition, $\int \Phi d\nu_n~\to~\int \Phi d\mu$, and the support of every ergodic measure is internally chain recurrent,
so:
\begin{align*}
\htop(T,R_{\Phi}(a,\theta)) &\leq  \sup \{h_{\mu}(T): \int \Phi d\mu \in (a-\theta,a+\theta) \text{ and }\supp(\mu)\subseteq Z \\ & \text{ for some internally chain transitive set } Z\subseteq X \}\\
&\leq \sup \{h_{\mu} (T): \int \Phi d\mu \in (a-\theta,a+\theta)\text{ and }\supp(\mu)\subseteq Y\in C_c(T, X) \}.
\end{align*}

It remains to show that:
\begin{align*}
\sup&\{ h_{\mu}(T): \int \Phi d\mu \in (a-\theta,a+\theta), \mu \in \mathcal{M}_e(T) \} \\ \geq
&\sup \{ h_{\mu}(T): \int\Phi d\mu \in (a-\theta,a+\theta), \supp(\mu) \subset Y\in C_c(T,X)  \}.
\end{align*}

Take any $\eta>0$ and any invariant measure $\mu$ supported on some chain recurrent class in $X$ with $\int \Phi d\mu \in (a-\theta,a+\theta)$. Then by mentioned result of \cite{LO} we find an ergodic measure $\nu$ with $\int \Phi d\nu \in (a-\theta,a+\theta)$ and $h_{\nu}(T)>h_{\mu}(T)-\eta$ which completes the proof.
\end{proof}

Theorem~\ref{thm:v-} is motivated by result for maps with the specification property in \cite{Young}, in particular mixing maps with the shadowing property. The following example shows that it has also a chance to hold where there are numerous chain-recurrent classes. Of course a trivial example is provided by identity on Cantor set with any measure, but it would be nice to have a more sophisticated example.

\begin{ex}\label{exV}
	We are going to construct a map $f\colon [0,1]\to [0,1]$ such that $f$ has the shadowing property, infinitely many chain recurrent classes,
	Lebesgue measure $m$ on $[0,1]$ is not $f$-invariant, and $\mathcal{V}^-\neq \emptyset$. This shows that Theorem~\ref{thm:v-} can be satisfied
	also in the case when there are infinitely many chain-recurrent classes.
	
	Let $a_n=2^{-n}$ and $b_n=5a_{n+1}/4$ for $n=0,1,\ldots$. Let $\lambda_n\in (\sqrt{2},2]$ be any sequence such that tent map with slope $\lambda_n$
	has shadowing property. It equivalently means, that this map satisfies the so called linking property \cite{Chen}, and by results of \cite{Yorke} most of slopes in $(\sqrt{2},2]$ satisfies it.
	We also assume that $\lambda_0=2$, which is admissible, since full tent map has the shadowing property.
	Let $f\colon [0,1]\to [0,1]$ be defined as follows. On the interval $[b_n,a_n]$ the map $f$ is the tent map with slope $\lambda_n$ and on $[a_{n+1},b_n]$ we have affine map with $f(a_{n+1})=b_{n+1}$ and $f(b_{n})=b_n$. Note that $f(b_n)-f(a_{n+1})=b_n-{b_{n+1}}=5a_{n+1}/4$ which shows that $f$ has slope $5$ on each of the intervals $[a_{n+1},b_n]$.
	This shows that $f$ is Lipschitz with Lipschitz constant $5$, and so if $m$ is Lebesgue measure, then for every $x$ and every $\eps>0$ we have:
	$$
	m(B_n(x,\eps))\geq \eps 5^{-n}=\eps e^{-n \ln (5)}.
	$$
	Then clearly $\mathcal{V}^{-}$ contains any function $\xi\geq \ln(5)$ but of course contains other functions as well.
	
	To see that $f$ has shadowing, it is enough to apply Theorem~7 from \cite{Chris}. Simply, let $\pi_{n}\colon [a_{n+1},a_0]\to [a_n,a_0]$ be defined for $n>0$
	by $\pi_n(x)=\max\{x,a_{n}\}$. In other words $\pi_n$ collapses $[a_{n+1},a_n]$ to the point $a_n$. Note that $f_n\circ \pi_n=\pi_n\circ f_{n+1}$
	where $f_n=f|_{[a_n,a_0]}$. Also each $f_n$ has shadowing property, as it is piecewise linear and by definition it has linking property (since $f$ on each $[a_{b},a_n]$ has it).
	Coordinate-wise action of $f_n$ on the inverse limit $\varprojlim(\pi_n, [a_n,a_0])$ has the shadowing property. But it is also easy too check that by the definition it is conjugate with $(f,[0,1])$.
\end{ex}

\section{Acknowledgements}
M. Fory\'{s}-Krawiec was supported from ESF "Strengthening the university's scientific capacities" (No. CZ.02.2.69/0.0/0.0/16\_027/0008472). 
P. Oprocha was partially supported by the Faculty of Applied Mathematics AGH UST statutory tasks within subsidy of Ministry of Science and Higher Education and by NPU II project LQ1602 IT4Innovations excellence
in science  and project lRP201824 ``Complex topological structures''.
J. Kupka was also supported by the project lRP201824 ``Complex topological structures''.
X. Tian was supported by National Natural Science Foundation of China (grant no. 11671093).


\begin{thebibliography}{99}
\bibitem{BSch} L. Barreira, J. Schmeling \textit{Sets of non-typical points have full topological measure and full Hausdorff dimension}, Israel Journal of Mathematics, \textbf{116} (2000) pp. 29-70,
\bibitem{Bow1} R. Bowen, \textit{Entropy-expansive maps}, Trans. Amer. Math. Soc., \textbf{164} (1972), pp. 323-331,
\bibitem{Bow2} R. Bowen, \textit{Periodic points and measures for Axiom A diffeomorphisms}, Trans. Amer. Math. Soc., \textbf{154} (1971), pp. 377-397,
\bibitem{Bow} R. Bowen, \textit{Topological entropy for noncompact sets}, Trans. Amer. Math. Soc. 184 (1973), 125-136,
\bibitem{BK} M. Brin, A. Katok, \textit{On local entropy}, Geom. Dyn. Springer Lecture Notes, 1007 (1983), 30-38,
\bibitem{Chen} L. Chen, \textit{Linking and the shadowing property for piecewise monotone maps}, Proc. Amer. Math. Soc. 113 (1991), 251–263,
\bibitem{ChTS} L. Chen, K. Tassilo, L. Shu \textit{Topological entropy for divergence points}, Ergod. Th. Dynam. Sys., \textbf{25} (2005), 1173-1208,
\bibitem{DOT} Y. Dong, P. Oprocha, X. Tian \textit{On the irregular points for systems with the shadowing property}, Erg. Theory and Dyn. Sys., \textit{38} (2018), no. 6, 2208-2131,
\bibitem{Chris} C. Good, J. Meddaugh, \textit{Shifts of finite type as fundamental objects in the theory of shadowing}, Preprint, 2017, arXiv:1702.05170,
\bibitem{FFW} A. Fan, D. Feng, J. Wu, \textit{Recurrence, dimensions and entropy}, J. London Math. Soc., \textbf{64} (2001), 229-244,
\bibitem{HSZ} M. W. Hirsch, H. L. Smith, X. Q. Zhao \textit{Chain transitivity and strong repellors for semidynamical systems}, J. Dynam. Differential Equations 13 (2001) no. 1, 107-131,
\bibitem{K} A. Katok, \textit{Lyapunov exponents, entropy and periodic orbits for diffeomorphisms}, Publ. Math. I. H. E. S., \textbf{51} (1980), 137-174,
\bibitem{LO} J. Li, P. Oprocha, \textit{Properties of invariant measures in dynamical systems with the shadowing property}, Erg. Theory and Dyn. Sys. 38 (2018) no. 6, 2257-2294,
\bibitem{LW} J. Li, M. Wu, \textit{Generic property of irregular sets in systems satisfying the specification property}, Discrete Contin. Dyn. Sys., \textbf{34} (2014), 635-645,
\bibitem{Mooth} T. K. S. Moothathu, \textit{Implications of pseudo-orbit tracing property of continuous maps on compacta}, Top. Appl. \textbf{158} (2011), 2232-2239,
\bibitem{MoothO} T. K. S. Moothathu, P. Oprocha, \textit{Shadowing, entropy and minimal subsystems}, Monatsh. Math. \textbf{172} (2013), 357-378,
\bibitem{O1} L. Olsen, \textit{Divergence points of deformed empirical measures}, Math. Res. Lett., \textbf{9} (2002), 1-13,
\bibitem{O2} L. Olsen, \textit{Multifractal analysis of divergence points of deformed measure theoretical Birkhoff averages}, J. Math. Pures Appl., \textbf{82} (2003), 1591-1649,
\bibitem{Ow} L. Olsen, S. Winter, \textit{Normal and non-normal points of self-similar sets and divergence points of self-similar measures}, J. London Math. Soc., \textbf{67} (2003), 103-122,
\bibitem{Sig1} K. Sigmund \textit{Generic properties of invariant measures for Axiom A diffeomorphisms}, Invent. Math., \textbf{11} (1970), 99-109,
\bibitem{TV} F. Takens, E. Verbitskiy, \textit{On the variational principle for the topological entropy of certain non-compact sets}, Ergodic Theory Dynam. Sys., \textbf{23} (2003), no.1, 317-348,
\bibitem{Tian2017}  X. Tian, \textit{Topological Pressure for the Completely Irregular Set of Birkhoff Averages}, Discrete Contin. Dyn. Sys.,  \textbf{37} (2017), 2745-2763.
\bibitem{Thom} D. J. Thompson, \textit{Irregular sets, the $\beta$-transformation and the almost specification property}, Trans. Amer. Math. Soc. 364 (2012), no. 10, 5395-5414,
\bibitem{Yorke}  E. M. Coven, I. Kan, and J. A. Yorke, \textit{Pseudo-orbit shadowing in the family of tent maps}, Trans. Amer. Math. Soc., 308 (1988), 227–241,
\bibitem{Young} L. Young, \textit{Some large deviation results for dynamical systems}, Trans. Amer. Math. Soc., 318 (1990), No. 2, 525-543
\end{thebibliography}
\end{document}